\documentclass[letterpaper, 10 pt, conference]{ieeeconf}
\IEEEoverridecommandlockouts
\usepackage{graphicx}
\usepackage{amsmath} 
\usepackage{amssymb}
\usepackage{dsfont}
\usepackage{mathtools}
\usepackage{stmaryrd}
\usepackage{tikz}
\usepackage{centernot}
\usepackage[compress]{cite}

\DeclareRobustCommand\redline{\raisebox{1.3pt}{\tikz{\draw[-,red,line width = 0.9pt](0,0) -- (5mm,0);}}}

\DeclareRobustCommand\blackcirc{\raisebox{-1.3mm}{\tikz{\draw[-,black,line width = 0.9pt](0,0) -- (5mm,0);\node[] at (2.5mm,-.1mm) {\large $\circ$};}}}

\DeclareRobustCommand\blackd{\raisebox{-1.3mm}{\tikz{\draw[-,black,line width = 0.9pt](0,0) -- (5mm,0);\node[] at (2.5mm,-.1mm) {\large $\diamond$};}}}

\newtheorem{theorem}{\bf Theorem}
\newtheorem{proposition}{\bf Proposition}

\newtheorem{corollary}{\bf Corollary}[section]
\newtheorem{assumption}{\bf Assumption}[section]
\newtheorem{remark}{\bf Remark}[section]
\newtheorem{lemma}{\bf Lemma}[section]
\newtheorem{definition}{\bf Definition}[section]

\makeatletter
\renewcommand*{\@opargbegintheorem}[3]{\trivlist
  \item[\hskip \labelsep{\bfseries #1\ #2}] \textbf{(#3)}\ \itshape}
\makeatother

\newcommand{\X}{\mathbb{X}}
\newcommand{\N}{\mathbb{N}}
\newcommand{\E}{\mathbb{E}}
\newcommand{\R}{\mathbb{R}}
\newcommand{\B}{\mathcal{B}}
\newcommand{\A}{\mathcal{A}}
\newcommand{\Q}{\mathbb{Q}}

\renewcommand{\O}{\mathbb{O}}
\renewcommand{\P}{\mathbb{P}}
\newcommand{\Pc}{\mathcal{P}}
\renewcommand{\d}{\mathrm{d}}
\usepackage[shortcuts]{extdash}

\newcommand\myeq[1]{\stackrel{\mathclap{\normalfont\mbox{\scriptsize #1}}}{=}}

\newcommand{\revise}[1]{{\color{black}#1}}

\allowdisplaybreaks

\newcommand*{\LONGVERSION}{}

\begin{document}

\title{\LARGE \bf Distributionally-Robust Optimization with Noisy Data \\ for Discrete Uncertainties Using Total Variation Distance}

\author{Farhad Farokhi
\thanks{F. Farokhi is with the Department of Electrical and Electronic Engineering, The University of Melbourne, Australia. This work has been financially supported, in part, by a grant from the Australian Research Council (ARC) DP210102454.}
}

\maketitle

\begin{abstract} 
Stochastic programs, where uncertainty distribution must be inferred from \textit{noisy} data samples, are considered. They are approximated with distributionally\-/robust optimizations that minimize the worst-case expected cost over ambiguity sets, i.e., sets of distributions that are sufficiently compatible with  observed data. The ambiguity sets capture probability distributions whose convolution with the noise distribution is within a ball centered at the empirical noisy distribution of data samples parameterized by total variation distance. Using the prescribed ambiguity set, the solutions of the distributionally\-/robust optimizations converge to the solutions of the original stochastic programs when the number of the data samples grow to infinity. Therefore, the proposed distributionally\-/robust optimization problems are \textit{asymptotically consistent}. The distributionally\-/robust optimization problems can be cast as \textit{tractable optimization problems}.
\end{abstract}

\section{Introduction}
In this paper, we consider a single-stage stochastic program of the form $\inf_{x\in\X}\E^{\P}[h(x,\xi)]$,
where $x\in\R^n$ is a decision variable that must be determined to minimize the expected cost $\E^{\P}[h(x,\xi)]$ while $\xi\in\R^m$ is a random uncertainty with distribution $\P$. Distribution $\P$ is not available and must be inferred from data samples. Merely relying on the empirical distribution evaluated using data samples, instead of the original distribution $\P$, can result in disappointing outcomes. This is known colloquially as the ``optimizer's curse'' and is caused by over-fitting~\cite{michaud1989markowitz}. One way to avoid this concern is to instead consider a set of distributions that are sufficiently compatible with the observed data, known as an ambiguity set $\mathcal{P}$, and minimizes the worst-case expected cost over the ambiguity set $\sup_{\Q\in\mathcal{P}}\E^{\Q}[h(x,\xi)]$. This approach is known as distributionally\-/robust optimization~\cite{mohajerin2018data}.

Distributionally-robust optimization dates back to  ambiguity-averse or robust news-vendor problem~\cite{scarf1958min}. However, there has been a recent revival in the field due to computationally-favourable reformulations~\cite{ben2013robust,delage2010distributionally,erdougan2006ambiguous,hu2013kullback,mohajerin2018data,pflug2007ambiguity,wozabal2012framework,bartl2021sensitivity}. These results mostly differ on how they construct the ambiguity sets using, e.g., moment constraints~\cite{delage2010distributionally}, the Prohorov metric~\cite{erdougan2006ambiguous}, the Kullback-Leibler divergence~\cite{hu2013kullback}, and the Wasserstein metric~\cite{wozabal2012framework, mohajerin2018data}. 

Distributionally-robust optimization  has however mostly  focused on noiseless data~\cite{ben2013robust,delage2010distributionally,erdougan2006ambiguous,hu2013kullback,mohajerin2018data,pflug2007ambiguity,wozabal2012framework}, i.e., high-quality independently and identically realized samples from distribution $\P$ are required. Noting that we sometimes only have access to noisy data samples, there has been some recent investigations into situations where the data samples are noisy~\cite{van2021efficient,farokhi2022distributionally}. For instance, data samples may be intentionally corrupted by privacy-preserving noise~\cite{abowd2018us}. Alternatively, data can be truly noisy because of inherent instrumentation uncertainty or noise. The problem with noisy data is that, even with infinitely-many data points, the empirical distribution does not coverage to the original distribution $\P$. The empirical distribution instead converges to $\P'$, which is the outcome of convolution of the original distribution $\P$ and the noise distribution $\O$~\cite{van2021efficient,farokhi2022distributionally}. \revise{One way to pose distributionally\-/robust optimization in the presence of noisy data is to expand the ambiguity sets, by setting the radius of the uncertainty ball around the empirical noisy distribution large enough, to contain the noiseless distribution~\cite{guo2022partial,farokhi2022distributionally}. This however can create unnecessary conservatism because we have to deal with the worst-case expected cost over a large ambiguity set that contains distributions that may behave differently from the original noiseless distribution. Such ideas effectively dictate that the radius of the ambiguity set must remain non-trivially large even in the big data regime. Therefore, these method can only provide useful solutions and guarantees in the small noise regime, i.e., when the variance or entropy of the noise is relatively small, and thus enlargement of the ambiguity set is minimally conservative. As opposed, in Section~\ref{sec:concentration} of this paper, we observe that the radius of the ambiguity set converges to zero as more samples are gathered. Therefore, the method of this paper is asymptotically less conservative.}  Another way is to use other metrics for constructing the ambiguity set~\cite{van2021efficient}. The alternatives however still suffer from conservatism discussed above. \revise{None of} these studies consider the specific way that the noise distribution changes the original distribution, i.e., through convolution. In fact, in the large data regime, it is reasonable to expect that the effect of the noisy measurements is negligible. This is because we can always remove the effect of the noise by density deconvolution~\cite{carroll1988optimal,farokhi2020deconvoluting} even if the noise is significant, i.e., it has large variance or entropy. 

An important aspect of obtaining good results in distributionally\-/robust optimization is to appropriately select  the ambiguity set. The ambiguity set must be large  enough to contain the original density $\P$ while it must be small enough to not make the results conservative. The solution of $\sup_{\Q\in\mathcal{P}}\E^{\Q}[h(x,\xi)]$ must remain reasonably close to $\E^{\P}[h(x,\xi)]$. Otherwise, we incentivize overly conservative decisions by considering distributions $\Q\in\mathcal{P}$ that are far from reality $\P$. The ambiguity set must also be easily reconstructable from the data samples to ensure computationally tractable reformulation of the distributionally\-/robust optimization problem. In this paper, we define the ambiguity set by considering a set of probability distributions whose convolution with the known noise distribution remains appropriately close to the empirical noisy distribution of the data samples in the sense of the total variation distance. We prove that, using the prescribed ambiguity set, the solution of the distributionally\-/robust optimization converges to the solution of the original stochastic program when the number of the data samples grows to infinity. Therefore, the distributionally\-/robust optimization problem is asymptotically consistent. To prove this result, we need to assume that the distribution of the noise is uniformly diagonally dominant. Finally, we show that the robust optimization problem can be cast as a tractable convex optimization problem.
 
The rest of the paper is organized as follows. First, we finish this section by presenting some useful notations. Subsequently, in Section~\ref{sec:data_driven}, we formally define distributionally\-/robust optimization based on noisy data samples. In Section~\ref{sec:concentration}, we consider asymptotic properties of the distributionally\-/robust optimization, such as asymptotic consistency, using concentration bounds for learning discrete distributions. We re-cast the distributionally\-/robust optimization as a tractable convex optimization problem in Section~\ref{sec:solving}. Finally, we present some numerical results in Section~\ref{sec:numerical} and conclude the paper in Section~\ref{sec:conclusions}.

\section*{Notation}
For any set $\A$, the \revise{cardinality} of the set is denoted by $|\A|$. For any finite set $\A$, i.e., $|\A|<\infty$, $\Delta(\A)$ denotes the probability simplex on $\A$. The product of two probability distributions $\P_1\in\Delta(\Xi_1)$ and $\P_2\in\Delta(\Xi_2)$ is the distribution $\P_1\otimes\P_2\in\Delta(\Xi_1\times\Xi_2)$. The $N$-fold product of a distribution $\P\in\Delta(\Xi)$ is denoted by $\P^N\in\Delta(\Xi^N)$. The total variation distance between any two distributions $\P_1,\P_2\in\Delta(\Xi)$ is
\begin{align*}
	\d_{\rm TV}(\P_1,\P_2):=&\sup_{\A\subseteq \Xi} (\P_1(\A)-\P_2(\A))\\
	=&\frac{1}{2}\sum_{\xi\in\Xi} |\P_1(\xi)-\P_2(\xi)|\in[0,1].
\end{align*}

\section{Data-Driven Programming} \label{sec:data_driven}
Consider the stochastic program
\begin{align}\label{eqn:stochastic_program}
	J^\star:=\inf_{x\in\X}\left\{\E^{\P}[h(x,\xi)]=\sum_{\xi\in\Xi}h(x,\xi)\P(\xi ) \right\},
\end{align}
with feasible set $\X\subseteq \R^n$, discrete/finite uncertainty set $\Xi\subseteq\R^m$ (i.e., $|\Xi|<\infty$), and loss function $h:\X\times\Xi\rightarrow\R$. The loss function $h$ depends on the decision vector $x\in\R^n$ and the random variable $\xi\in\R^m$, whose distribution $\P$ is supported on $\Xi$, i.e., $\P\in\Delta(\Xi)$. The distribution $\P$ is  \textit{unknown}. Therefore  problem~\eqref{eqn:stochastic_program} cannot be solved exactly. Although $\P$ is unknown, it is partially observable through a finite set of $N$ independently and identically distributed (i.i.d.) \textit{noisy} samples 
\begin{align}
	\xi'_i\sim \O(\cdot|\xi_i),\quad \xi_i\sim\P, \quad i\in\{1,\dots,N\},
\end{align}
where $\xi_i\sim \P$ are i.i.d.~samples from $\P\in\Delta(\Xi)$ while $\xi'_i$ are noisy observations of $\xi_i$ realized according to the conditional distribution $\O(\cdot|\xi_i)\in\Delta(\Xi')$. The marginal distribution of the noisy observations is given by
\begin{align} \label{eqn:noisy_unnoisy}
	\P'(\xi')=\sum_{\xi\in\Xi} \O(\xi'|\xi) \P(\xi), \quad \forall \xi'\in\Xi'.
\end{align}
Note that support set of $\P'$, which is $\Xi'$, may not necessarily be equal to the support set of $\P$, which is $\Xi$. Nonetheless, we assume that $\Xi',\Xi\subseteq\R^m$. This assumption is not strictly necessary, however, it simplifies the narrative without substantial conservatism (up to relabeling the elements in $\Xi'$). For instance, noisy data based on additive noise satisfies this condition.  For the sake of brevity, we write that
\begin{align}
	\P'=\O\star \P.
\end{align}
Note that, here, we have opted for the convolution notation `$\star$'  because the relationship in~\eqref{eqn:noisy_unnoisy} \revise{is visually similar to} discrete convolution of the probability mass functions, particularly when the noise is additive\footnote{
	\revise{When dealing with additive noise, i.e., $\xi'_i=\xi_i+n_i$, the marginal distribution of the noisy observations in~\eqref{eqn:noisy_unnoisy} can be rewritten as $\P'(\xi')=\sum_{\xi\in\Xi} \O(\xi'-\xi) \P(\xi)$ as, by slight abuse of notation, $\O(\xi'|\xi)=\O(\xi'-\xi)$. This implies that $\P'$ is equal to convolution of $\O$ and $\P$.}
}.  

\revise{
\begin{remark}[Known Noise Distribution]
	We assume that the distribution of the noise $\O$ is known. The motivation for this is twofold. First, in privacy-preserving applications (e.g., the example in Section~\ref{sec:numerical}), the distribution of the noise, which is a function of the privacy budget and privacy-preserving mechanism, is often publicly known to improve transparency and accountability and to also allow for post processing~\cite{Dwork_Kohli_Mulligan_2019}. Furthermore, in sensing and instrumentation, the distribution of the noise is usually publicly shared in data-sheets while the distribution of the underlying variable is oft unknown. When the distribution of the noise is not known two approaches can be used. We can either use repeated measurements to estimate the distribution of the noise~\cite{101214009053607000000884}, which is not possible in privacy-preserving applications as repeated queries/responses erode privacy, or follow a worst-case adversarial approach~\cite{bennouna2022holistic}, which may be conservative.
\end{remark}

\begin{remark}[Finite Uncertainty Set] The assumption that the uncertainty set $\Xi$ is finite is natural in some cases. For instance, the data could be inherently discrete, such as categorical attributes~\cite{wang2014coupled} and finite state spaces~\cite{hoshino2021probability}. In other instances, communication constraints or post-processing can render the data discrete~\cite{9511566,wu2021l2}.
\end{remark}

\begin{remark}[Impact of Density Deconvolution] 
    Combining density deconvolution~\cite{farokhi2020deconvoluting} and techniques used for developing computationally-efficient distributionally-robust optimization~\cite{mohajerin2018data} is non-trivial. This is because after deconvolution, the empirical density function is no longer composed of delta functions, which breaks down an important step in the proof of Theorem 4.2 in~\cite{mohajerin2018data}.
\end{remark}

}

We denote the training dataset composed of the noisy samples by $\Xi'_{N}:=\{\xi'_i\}_{i=1}^N$. A data-driven solution for problem~\eqref{eqn:stochastic_program} is a feasible decision $\hat{x}_N\in\X$ that is constructed from the training dataset $\Xi'_{N}$. The \textit{out-of-sample performance} of $\hat{x}_N$ is defined as $\E^\P[h(\hat{x}_N,\xi)],$ which is the expected cost of the data-driven solution $\hat{x}_N\in\X$ for a new sample $\xi$ from $\P$, which is independent of the training dataset. Since $\P$ is unknown, the out-of-sample performance cannot be evaluated explicitly in practice. Therefore, we would like to establish performance guarantees for the out-of-sample performance. By construction, because of the feasibility of $\hat{x}_N\in\X$, we know that
\begin{align*}
	\E^\P[h(\hat{x}_N,\xi)]\geq J^\star,
\end{align*}
where $J^\star$ is the optimal solution in~\eqref{eqn:stochastic_program}. In line with the literature on distributionally\-/robust optimization~\cite{mohajerin2018data}, we are interested in out-of-sample performance bounds:
\begin{align*}
	{\P'}^N\{\Xi'_N:\E^\P[h(\hat{x}_N,\xi)]\leq \hat{J}_N\}\geq 1-\beta,
\end{align*}
where $\hat{J}_N$ is an upper bound that potentially depends on the training dataset and $\beta\in(0,1)$ is a significance or confidence parameter. Note that the dataset $\Xi'_{N}$ is a random variable governed by the $N$-fold product\footnote{\revise{Note that ${\Xi'}^N$ denotes $N$-fold Cartesian product of set $\Xi'$ with itself, i.e., ${\Xi'}^N=\Xi'\times\dots\times\Xi'$, while $\Xi'_{N}$ denotes the set of noisy data samples, i.e.,  $\Xi'_{N}=\{\xi'_i\}_{i=1}^N$. Therefore, by definition, $\Xi'_{N}\subseteq {\Xi'}^N$. The same distinction also holds for $\Xi^N$ and $\Xi_N$.}} distribution ${\P'}^N\in\Delta({\Xi'}^N)$. 

One way to solve this problem is to compute the discrete empirical probability distribution $\hat{\P}'_N\in\Delta(\Xi')$:
\begin{align} \label{eqn:noisy_data_empirical_distribution}
	\hat{\P}'_N(\xi)=\frac{1}{N}\sum_{i=1}^N \delta[\xi'_i-\xi],
\end{align}
where $\delta:\R^m\rightarrow\R$ is the Kronecker delta function, i.e., $\delta[x]=1$ if $x=0$ and $\delta[x]=0$ otherwise. This amounts to approximating the stochastic program in~\eqref{eqn:stochastic_program} with the noisy sample-average approximation (NSAA) problem:
\begin{align}
	\hat{J}_{\rm NSAA}:=\inf_{x\in\X}\left\{\E^{\hat{\P}'_N}[h(x,\xi)]=\frac{1}{N}\sum_{i=1}^N h(x,\xi'_i) \right\}.
\end{align}
However, it should be noted that $ \lim_{N\rightarrow\infty}\E^{\hat{\P}'_N}[h(x,\xi)]\myeq{a.s.}\E^{\P'}[h(x,\xi)]\neq \E^{\P}[h(x,\xi)].$ 
This is caused by the noisy nature of the samples. In this paper, we address this problem by explicitly considering the effect of noisy measurements. We particularly use $\Xi'_N$ to create an ambiguity set $\hat{\Pc}_N\subseteq\Delta(\Xi)$ containing all distributions that could have generated the noiseless samples with high confidence. This ambiguity set enables us to define the distributionally\-/robust optimization problem:
\begin{align} \label{eqn:DRO}
	\hat{J}_{\rm DRO} (\hat{\Pc}_N):=\inf_{x\in\X}\sup_{\Q\in\hat{\Pc}_N}\E^{\Q}[h(x,\xi)].
\end{align}
If the optimal solution to~\eqref{eqn:DRO} exists and is attained for some element of $\X$, we denote the solution with $\hat{x}_{\rm DRO}(\hat{\Pc}_N)$. We drop the reference to $\hat{\Pc}_N$ and use $\hat{J}_{\rm DRO} $ and $\hat{x}_{\rm DRO} $ instead of $\hat{J}_{\rm DRO} (\hat{\Pc}_N)$ and $\hat{x}_{\rm DRO} (\hat{\Pc}_N)$, respectively, when the ambiguity set is clear from the context. 

\revise{
\begin{remark}[Existence of Solution]
 The inner optimization problem in~\eqref{eqn:DRO}, i.e., $\sup_{\Q\in\hat{\Pc}_N}\E^{\Q}[h(x,\xi)]$, possesses a finite supremum if the cost function $h(x,\xi)$ is bounded and the ambiguity set $\hat{\Pc}_N$ is compact, which holds for the ambiguity sets defined using total variation distance in Section~\ref{sec:concentration}. Existence of solution to the outer problem, i.e., $\inf_{x\in\X} (\sup_{\Q\in\hat{\Pc}_N}\E^{\Q}[h(x,\xi)])$, is more subtle. The solution exists and is attained whenever $\sup_{\Q\in\hat{\Pc}_N}\E^{\Q}[h(x,\xi)]$ is continuous in $x$ (e.g., if the solution to the inner problem is unique, and the cost function is bounded and continuous) and the feasible set $\X$ is compact. Note that existence of solutions does not imply computational feasibility of finding one. The latter requires extra assumptions, e.g., convexity. 
\end{remark}
}

In what follows, we construct $\hat{\Pc}_N$ as a ball around the empirical distribution~\eqref{eqn:noisy_data_empirical_distribution} with respect to the total variation distance with an additional constraint based on density convolution to account for the noisy nature of the data. 

\section{Concentration Bounds on Empirical Discrete Distributions}
\label{sec:concentration}
The following concentration bounds provide the basis for establishing finite sample guarantees that we use to develop the ambiguity set $\hat{\Pc}_N\subseteq\Delta(\Xi)$ in the distributionally\-/robust optimization framework of~\eqref{eqn:DRO}. 

\begin{theorem}[{{\hspace{-.1mm}\cite[Theorem~1]{canonne2020short}}}] \label{tho:TV_concentration} For the empirical distribution in~\eqref{eqn:noisy_data_empirical_distribution}, ${\P'}^N\{\d_{\rm TV}(\P',\hat{\P}'_N)\leq \varepsilon\}\geq 1-\alpha$ if $
	N\geq {\max\{|\Xi|,2\ln(2/\alpha)\}}/{\varepsilon^2}.$
\end{theorem}

Let us define total-variation ball $\B'_{\rm TV,\varepsilon}(\hat{\P}'_N):=\{\Q:\d_{\rm TV}(\Q,\hat{\P}'_N)\leq \varepsilon\}$ and 
\begin{align}
	\varepsilon_{\rm TV}(\alpha):=\sqrt{\frac{\max\{|\Xi|,2\ln(2/\alpha)\}}{N}}. 
\end{align}
The following corollary follows from Theorems~\ref{tho:TV_concentration}.

\begin{corollary} \label{cor:distribution_balls}
		${\P'}^N\{\P'\in \B'_{\rm TV,\varepsilon_{\rm TV}(\alpha)}(\hat{\P}'_N)\}\geq 1-\alpha$.
\end{corollary}

Note that, so far, we have been focused on concentration bounds for learning $\P'$. However, to solve~\eqref{eqn:stochastic_program}, we must learn $\P$. Let us define sets
	\begin{align} \label{eqn:def_new_set}
		\B_{\rm TV,\varepsilon}(\hat{\P}'_N):=\{\Q:  \O\star\Q \in \B'_{\rm TV,\varepsilon}(\hat{\P}'_N) \}.
	\end{align}
The following corollary immediately follows from Corollary~\ref{cor:distribution_balls} and the definition of the ambiguity set in~\eqref{eqn:def_new_set}. 

\begin{corollary} \label{cor:distribution_slices}
 	${\P'}^N\{\P\in \B_{\rm TV,\varepsilon_{\rm TV}(\alpha)}(\hat{\P}'_N)\}\geq 1-\alpha$.
\end{corollary}

\revise{
\begin{remark}[Total Variation Distance vs. Other Probability Metrics]
	Various probability metrics, such as Kullback–Leibler divergence~\cite{hu2013kullback} and Wasserstein distance~\cite{mohajerin2018data}, are used for defining the ambiguity sets in distributionally-robust optimization. The use of total variation distance in this paper gives rise to a linear programming problem for computing the worst-case distribution (see  Theorem~\ref{tho:worst_case}), which is computationally favorable. The use of Kullback–Leibler divergence would have resulted in a convex nonlinear program. Also note that the total variation distance is an optimal transportation distance with an indicator cost function~\cite{villani2008optimal}. To use other optimal transport distances, such as the Wasserstein distance, we must endow the finite uncertainty sets $\Xi$ and $\Xi'$ with distances. For categorical sets that are not subsets of metric spaces,  e.g., $\{\mbox{male},\mbox{ female}\}$, distances can be artificial.
\end{remark}
}

\begin{theorem}[Out-of-Sample Performance]
	Assume that $\hat{J}_{\rm DRO}$ and $\hat{x}_{\rm DRO}$ denote the optimal value and an optimizer of the distributionally\-/robust optimization problem~\eqref{eqn:DRO}, where the ambiguity set is $\hat{\Pc}_N=\B_{\rm TV,\varepsilon_{\rm TV}(\alpha)}(\hat{\P}'_N)$. Then,
	\begin{align*}
		{\P'}^N\{\Xi'_N:\E^\P[h(\hat{x}_{\rm DRO},\xi)]\leq \hat{J}_{\rm DRO}\}\geq 1-\alpha.
	\end{align*}
\end{theorem}

\begin{proof}
	Corollary~\ref{cor:distribution_slices} shows that $\P\in \hat{\Pc}_N$ with probability of at least $1-\alpha$. Therefore, with probability of at least $1-\alpha$, $\E^\P[h(\hat{x}_{\rm DRO},\xi)]\leq \sup_{\Q\in \hat{\Pc}_N}\E^\Q[h(\hat{x}_{\rm DRO},\xi)]\linebreak =\hat{J}_{\rm DRO}$.
\end{proof}

An important property for a stochastic estimator with access to random samples is consistency, i.e., the property that, as the number of samples increases towards infinity, the resulting sequence of estimates converges in some sense (convergence in probability or almost sure convergence) to the true solution~\cite[\S~2.3]{akahira2012asymptotic}. This has been a particularly sought-after property in distributionally\-/robust optimization~\cite{mohajerin2018data}. To prove consistency, we make the following assumption regarding the conditional probability of the noisy measurements. 

\begin{assumption} \label{assum:diagonally_dominant}
	For $\Xi'=\Xi$, $\O$ is uniformly diagonally dominant if $\min_{\xi\in\Xi}\O(\xi|\xi)>|\Xi| \max_{\xi,\xi'\in\Xi,\xi\neq \xi'}\O(\xi'|\xi)$.
\end{assumption}

Assumption~\ref{assum:diagonally_dominant} focuses on conditional distributions that are uniformly diagonally dominant. This is a slightly stronger notion than diagonal dominance, i.e., $\O(\xi|\xi)>\sum_{\xi'\neq \xi}\O(\xi'|\xi)$. Diagonal dominance is a powerful tool for analysis in linear algebra~\cite[\S~2]{johnston2021advanced} and probability~\cite{jiang2021consistent}. Assumption~\ref{assum:diagonally_dominant} essentially requires that the data is not heavily perturbed by the noise.

\begin{theorem} \label{tho:consistency}
	Assume that $\alpha_N\in(0,1)$, for all $N\in\N$, be such that $\sum_{N=1}^\infty \alpha_N<\infty$ while $\lim_{N\rightarrow\infty} \epsilon_{\rm TV}(\alpha_N)=0$. Furthermore, assume that $\hat{J}_{\rm DRO}$ and $\hat{x}_{\rm DRO}$ denote the optimal value and an optimizer of the distributionally\-/robust optimization problem~\eqref{eqn:DRO}, where the ambiguity set is $\hat{\Pc}_N=\B_{\rm TV,\varepsilon_{\rm TV}(\alpha)}(\hat{\P}'_N)$. Then, under Assumptions~\ref{assum:diagonally_dominant}, 
	\begin{itemize}
		\item If there exists $L\geq 0$ such that $|h(x,\xi)|\leq L$ for all $(x,\xi)\in\X\times\Xi$, then $\lim_{N\rightarrow\infty}\hat{J}_{\rm DRO}\myeq{a.e.}J^\star$.
		\item If there exists $L\geq 0$ such that $|h(x,\xi)|\leq L$ for all $(x,\xi)\in\X\times\Xi$, $\X$ is closed, and $h(x,\xi)$ is lower semi-continuous in $x$ for every $\xi\in\Xi$, then any accumulation point of $\{\hat{x}_{\rm DRO}\}_{N\in\N}$ is almost surely an optimal solution for~\eqref{eqn:stochastic_program}.
	\end{itemize}
\end{theorem}

\ifdefined\SHORTVERSION
\begin{proof}
The proof of this theorem is inspired by~\cite[Theorem~3.6]{mohajerin2018data}. Several aspects are however changed to accommodate noisy measurements, which was not considered in that paper.  \revise{The detailed proof is moved to an online report~\cite{Farokhi_report_2023} due to space constraints.}
\end{proof}

\fi 

\ifdefined\LONGVERSION
\begin{proof} The proof of this theorem is inspired by~\cite[Theorem~3.6]{mohajerin2018data}. Several aspects are however changed to accommodate noisy measurements, which was not considered in that paper. 
	
	As $\hat{x}_{\rm DRO}\in\X$,  $J^\star\leq \E^{\P}[h(\hat{x}_{\rm DRO},\xi)].$ Therefore, Corollary~\ref{cor:distribution_slices} implies that $ {\P'}^N\{J^\star \leq \E^{\P}[h(\hat{x}_{\rm DRO}, \xi)] \leq \hat{J}_{\rm DRO} \}  \geq {\P'}^N\{\P \in \hat{\Pc}_N\} \geq 1-\alpha_N.$ The Borel-Cantelli Lemma~\cite[Theorem~2.18]{kallenberg2002foundations} in conjunction with that $\sum_{N}\alpha_N<\infty$ implies that
	\begin{align}
		{\P'}^\infty\{J^\star \leq \E^{\P}[h(&\hat{x}_{\rm DRO}, \xi)] \leq \hat{J}_{\rm DRO}
		\nonumber\\&\mbox{ for all sufficiently large } N\} =1.\label{eqn:proof:6}
	\end{align}
	In what follows, we prove that $\limsup_{N\rightarrow\infty} \hat{J}_{\rm DRO} \leq J^\star$ with probability one. To do so, note that, for any $\Q\in\B_{\rm TV,\varepsilon_{\rm TV}(\alpha_N)}$, we have
	\begin{subequations}
		\begin{align}
			\E^{\Q}[h(x,\xi)]
			=&\sum_{\xi\in\Xi}h(x,\xi)\Q(\xi)\nonumber\\
			=&\sum_{\xi\in\Xi}h(x,\xi)\P(\xi)+\sum_{\xi\in\Xi}h(x,\xi)(\revise{\Q}(\xi)-\revise{\P}(\xi))\nonumber\\
			\leq& \E^{\P}[h(x,\xi)]+L\sum_{\xi\in\Xi}|\P(\xi)-\Q(\xi)| \label{eqn:proof:1}\\
			=&\E^{\P}[h(x,\xi)]+2L\d_{\rm TV}(\P,\Q) \label{eqn:proof:2}\\
			\leq &\E^{\P}[h(x,\xi)]\nonumber\\
                &+2Lc_0(\O)\d_{\rm TV}(\O\star\P,\O\star\Q)\label{eqn:proof:3}\\
			\leq &\E^{\P}[h(x,\xi)]+2Lc_0(\O)\varepsilon_{\rm TV}(\alpha_N),\label{eqn:proof:4}
		\end{align}
	\end{subequations}
	where~\eqref{eqn:proof:1} follows from that $|h(x,\xi)|\leq L$,~\eqref{eqn:proof:2} follows from the definition of the total variation distance,~\eqref{eqn:proof:3} follows from Lemma~\ref{lemma:useful} in the appendix, and finally~\eqref{eqn:proof:4} follows from that $\hat{\Pc}_N:=\B_{\rm TV,\varepsilon_{\rm TV}(\alpha)}(\hat{\P}'_N)$. Therefore, because $\lim_{N\rightarrow\infty}\varepsilon_{\rm TV}(\alpha_N)=0$, we get
	\begin{align}
		\limsup_{N\rightarrow\infty} \hat{J}_{\rm DRO}
		&\leq \limsup_{N\rightarrow\infty}  \sup_{\Q\in\hat{\Pc}_N}\E^{\Q}[h(x^*,\xi)]\nonumber\\
		&\leq \limsup_{N\rightarrow\infty}  (\E^{\P}[h(x^*,\xi)] + 2Lc_0(\O)\varepsilon_{\rm TV}(\alpha_N))\nonumber\\
		&\myeq{a.s.}  \E^{\P}[h(x^*,\xi)]=J^\star.\label{eqn:proof:5}
	\end{align}
	Combining~\eqref{eqn:proof:6} and~\eqref{eqn:proof:5} shows that $\lim_{N\rightarrow\infty}\hat{J}_{\rm DRO}\myeq{a.e.}J^\star$.
	
	Let $\bar{x}$ be an accumulation point of the sequence $\{\hat{x}_{\rm DRO}\}_{N\in\N}$. Evidently, $\bar{x}\in \X$ as $\X$ is closed. Note that, without loss of generality, we can assume that $\lim_{N\rightarrow\infty}\hat{x}_{\rm DRO}=\bar{x}$. This is because we can always select a subsequence of $\{\hat{x}_{\rm DRO}\}_{N\in\N}$ that meets this property. Therefore,
	\begin{subequations}
		\begin{align}
			J^\star 
			\leq& \E^{\P}[h(\bar{x},\xi)] \label{eqn:proof:10}\\
			\leq & \E^{\P}[\liminf_{N\rightarrow\infty}h(\hat{x}_{\rm DRO},\xi)] \label{eqn:proof:11}\\
			\leq &\liminf_{N\rightarrow\infty} \E^{\P}[h(\hat{x}_{\rm DRO},\xi)] \label{eqn:proof:12}\\
			\leq & \lim_{N\rightarrow\infty} \hat{J}_{\rm DRO} \label{eqn:proof:13}\\
			\myeq{a.e.} & J^\star \label{eqn:proof:14}
		\end{align}
	\end{subequations}
	where~\eqref{eqn:proof:10} follows from that $\bar{x}\in\X$,~\eqref{eqn:proof:11} follows from lower semi-continuity of $h(x,\xi)$ in $x$,~\eqref{eqn:proof:12} follows from the  Fatou's lemma~\cite[p.\,18]{lieb2001analysis},~\eqref{eqn:proof:13} follows from~\eqref{eqn:proof:6}, and finally~\eqref{eqn:proof:14} follows from our proof earlier that $\lim_{N\rightarrow\infty}\hat{J}_{\rm DRO}\myeq{a.e.}J^\star$.
\end{proof}
\fi

 \begin{remark}[Finite-Sample Convergence]
 	Under Assumption~\ref{assum:diagonally_dominant}, the proof of Theorem~\ref{tho:consistency} demonstrates that $ \hat{J}_{\rm DRO}-J^\star=\mathcal{O}(\varepsilon_{\rm TV}(\alpha_N))=\mathcal{O}(\ln^{1/2}(1/\alpha_N)N^{-1/2})$ when the ambiguity set is determined by the total variation distance, i.e., $\hat{\Pc}_N:=\B_{\rm TV,\varepsilon_{\rm TV}(\alpha)}(\hat{\P}'_N)$. Therefore, for fixed $\alpha_N=\alpha$, $ \hat{J}_{\rm DRO}-J^\star=\mathcal{O}(N^{-1/2})$. The dependency on $N$ seems to be order optimal, i.e., there seems to exist cost functions and distributions for which $\hat{J}_{\rm DRO}-J^\star=\Omega(N^{-1/2})$ even when the samples are not noisy; see, e.g.,~\cite[Proposition~1]{wibisono2012finite} for continuous distributions. An interesting direction for future research remains to prove the lower bound $\hat{J}_{\rm DRO}-J^\star=\Omega(N^{-1/2})$ for the exact problem formulation in this paper, i.e., noisy samples and discrete random variables.
 \end{remark}

\begin{remark}[Deconvolution]
	The importance of Assumption~\ref{assum:diagonally_dominant} remains to be fully investigated. In our proofs, this assumption is used to show that $\lim_{N\rightarrow\infty} \B_{\rm TV,\varepsilon_{\rm TV}(\alpha_N)}(\hat{\P}'_N)=\{\P\}$. Therefore, under appropriate conditions that ensure the uniqueness of the deconvolution, i.e., conditions under which it is guaranteed that $\{\Q\in\Delta(\Xi):\O\star \Q=\P'\}=\{\P\}$, the uniform diagonal dominance in Assumption~\ref{assum:diagonally_dominant} may not be necessary to ensure consistency. Asserting appropriate conditions and proving this statement remains an important direction for future research. 
\end{remark}

\section{Worst-Case Distributions} \label{sec:solving}
Motivated by the results of the previous section, we examine a generic worst-case expectation problem:
\begin{align} \label{eqn:generic_DRO}
	\sup_{\Q\in \B_{\rm TV,\varepsilon}(\hat{\P}'_N)} \E^{\Q}[\ell(\xi)].
\end{align}
For instance, $\ell(\xi)=h(x,\xi)$ for a fixed $x\in\X$. 

\begin{theorem} \label{tho:worst_case}
	The worst-case expectation problem in~\eqref{eqn:generic_DRO} equals
	\begin{align*}
		\begin{cases}
			\displaystyle \inf_{\scriptsize
			\begin{array}{c}
			(\lambda(\xi'))_{\xi'\in\Xi'}, \\ (\mu(\xi'))_{\xi'\in\Xi'}, \\
			r,t
			\end{array}
			} & \displaystyle r+2\varepsilon t+\sum_{\xi'\in\Xi}(\mu(\xi')-\lambda(\xi')) \hat{\P}'_N(\xi'),\\
			\displaystyle \hspace{8mm}\mathrm{s.t.} & \displaystyle \ell(\xi) +\sum_{\xi'\in\Xi}(\lambda(\xi') -\mu(\xi'))\O(\xi'|\xi)\leq r,\\
			& \displaystyle \lambda(\xi')+\mu(\xi')\leq t,\\
			& \displaystyle \lambda(\xi')\geq 0, \forall \xi'\in\Xi',\\
			& \displaystyle \mu(\xi')\geq 0, \forall \xi'\in\Xi'.
		\end{cases}
	\end{align*}
\end{theorem}

\ifdefined\SHORTVERSION
\begin{proof}
	\revise{The proof is moved to an online report~\cite{Farokhi_report_2023} due to space constraints.}
\end{proof}
\fi

\ifdefined\LONGVERSION
\begin{proof} Note that
	\begin{align*}
		&\sup_{\Q\in\B_{\rm TV,\varepsilon}(\hat{\P}'_N)} \E^{\Q}[\ell(\xi)]
		\\&\!=\!
		\begin{cases}
			\displaystyle \sup_{\Q\in\Delta(\Xi)} & \displaystyle \sum_{\xi\in\Xi} \ell(\xi) \Q(\xi),\\
			\displaystyle \mathrm{s.t.} & \displaystyle \sum_{\xi'\in\Xi'}|\hat{\P}'_N(\xi')-[\O\star\Q](\xi')|\leq 2\varepsilon,
		\end{cases}\\
		&\!=\!
		\begin{cases}
			\displaystyle \sup_{\Q\in\Delta(\Xi),s:\Xi'\rightarrow\R} & \displaystyle \sum_{\xi\in\Xi} \ell(\xi) \Q(\xi),\\
			\displaystyle \quad\quad\;\mathrm{s.t.} & \displaystyle
			|\hat{\P}'_N(\xi')-[\O\star\Q](\xi')|\leq s(\xi'),\\ 
			& \displaystyle \sum_{\xi'\in\Xi'}s(\xi')\leq 2\varepsilon,
		\end{cases}\\
		&\!=\!
		\begin{cases}
			\displaystyle \sup_{\Q\in\Delta(\Xi),s\in\mathcal{S}} & \displaystyle \sum_{\xi\in\Xi} \ell(\xi) \Q(\xi),\\
			\displaystyle \quad\;\;\mathrm{s.t.} & \displaystyle
			\hat{\P}'_N(\xi')-[\O\star\Q](\xi')\leq s(\xi'),\\ 
			&[\O\star\Q](\xi')-\hat{\P}'_N(\xi')\leq s(\xi'),
		\end{cases}
	\end{align*}
	where $\mathcal{S}:=\{s:\Xi'\rightarrow\R|\sum_{\xi'\in\Xi'}s(\xi')\leq 2\varepsilon\}$. 
	Using a standard duality argument, we obtain
	\begin{align*}
		&\sup_{\Q\in\B_{\rm TV,\varepsilon}(\hat{\P}'_N)} \E^{\Q}[\ell(\xi)]\\
		=&\sup_{\Q,s} \inf_{\lambda,\mu\geq 0}\Bigg[ \sum_{\xi\in\Xi} \ell(\xi) \Q(\xi) \\
		&\hspace{.3in}+\sum_{\xi'\in\Xi}\lambda(\xi')\Bigg(s(\xi')+\sum_{\xi\in\Xi}\O(\xi'|\xi)\Q(\xi)-\hat{\P}'_N(\xi')\Bigg)\Bigg]\\
		&\hspace{.3in}+\sum_{\xi'\in\Xi}\mu(\xi')\Bigg(s(\xi')-\sum_{\xi\in\Xi}\O(\xi'|\xi)\Q(\xi)+\hat{\P}'_N(\xi')\Bigg)\Bigg]\\
		=&\sup_{\Q,s} \inf_{\lambda,\mu\geq 0}\Bigg[ 	\sum_{\xi\in\Xi} \Q(\xi) \bigg(\ell(\xi)\!+\!\sum_{\xi'\in\Xi}(\lambda(\xi')\!-\!\mu(\xi'))\O(\xi'|\xi) \bigg) \\
		&\hspace{.6in}+\sum_{\xi'\in\Xi}(\lambda(\xi')+\mu(\xi')) s(\xi') \\
		&\hspace{.6in}+\sum_{\xi'\in\Xi}(\mu(\xi')-\lambda(\xi')) \hat{\P}'_N(\xi')\Bigg]\\
		=& \inf_{\lambda,\mu\geq 0}\sup_{\Q,s}\Bigg[ 	\sum_{\xi\in\Xi} \Q(\xi) \bigg(\ell(\xi)\!+\!\sum_{\xi'\in\Xi}(\lambda(\xi')\!-\!\mu(\xi'))\O(\xi'|\xi) \bigg) \\
		&\hspace{.6in}+\sum_{\xi'\in\Xi}(\lambda(\xi')+\mu(\xi')) s(\xi') \\
		&\hspace{.6in}+\sum_{\xi'\in\Xi}(\mu(\xi')-\lambda(\xi')) \hat{\P}'_N(\xi')\Bigg]\\	
		=&\inf_{\lambda,\mu\geq 0}\Bigg[\max_{\xi\in\Xi} \bigg(\ell(\xi) +\sum_{\xi'\in\Xi}(\lambda(\xi')-\mu(\xi'))\O(\xi'|\xi) \bigg)\\
		&\hspace{.6in}+2\varepsilon\max_{\xi'\in\Xi'}(\lambda(\xi')+\mu(\xi'))\\
		&\hspace{.6in}+\sum_{\xi'\in\Xi}(\mu(\xi')-\lambda(\xi')) \hat{\P}'_N(\xi')\Bigg],
	\end{align*}
	where equality holds due to strong duality because the optimization problem is convex (in fact it is a linear program) and Slater condition holds with $\Q'=\P'=\O\star\P$~\cite[\S~5.2.3]{boyd2004convex}. Therefore, we can reformulate~\eqref{eqn:generic_DRO} as the optimization problem in the statement of the theorem.
\end{proof}
\fi

Now, we are ready to leverage the result of Theorem~\ref{tho:worst_case} to compute the solution to \revise{the distributionally-robust optimization problem in~\eqref{eqn:DRO} with $\hat{\Pc}_N=\B_{\rm TV,\varepsilon}(\hat{\P}'_N)$.} In the next corollary, a computationally-friendly convex reformulation for this problem is provided. 

\begin{corollary} \label{cor:worst_case}
    The worst-case expectation problem \revise{in~\eqref{eqn:DRO} with $\hat{\Pc}_N=\B_{\rm TV,\varepsilon}(\hat{\P}'_N)$} equals
	\begin{align*}
		\begin{cases}
			\!\!\!\!\displaystyle \inf_{\scriptsize
				\begin{array}{c}
					(\lambda(\xi'))_{\xi'\in\Xi'}, \\ (\mu(\xi'))_{\xi'\in\Xi'}, \\
					r,t,x\in\X
				\end{array}
			} & \displaystyle r+2\varepsilon t+\sum_{\xi'\in\Xi}(\mu(\xi')-\lambda(\xi')) \hat{\P}'_N(\xi'),\\
			\!\!\!\!\displaystyle \hspace{8mm}\mathrm{s.t.} & \displaystyle h(x,\xi) +\sum_{\xi'\in\Xi}(\lambda(\xi') -\mu(\xi'))\O(\xi'|\xi)\leq r,\\
			& \displaystyle \lambda(\xi')+\mu(\xi')\leq t,\\
			& \displaystyle \lambda(\xi')\geq 0, \forall \xi'\in\Xi',\\
			& \displaystyle \mu(\xi')\geq 0, \forall \xi'\in\Xi'.
		\end{cases}
	\end{align*}
\end{corollary}

\revise{
\begin{remark}
    Note that~\eqref{eqn:DRO} involves two nested optimization problems (one maximization and one minimization) while the equivalent problem in Corollary~\ref{cor:worst_case} contains only a single minimization. This can significantly reduce the complexity of solving the problem. Furthermore, if $h(\cdot,\xi):\X\rightarrow\R$ is quasi-convex, the overall optimization problem in  Corollary~\ref{cor:worst_case} is convex because the cost function and the constraints are convex in all decision variables. Although the proof of  Corollary~\ref{cor:worst_case} may not require convexity of $h$, solving the optimization problem for non-convex $h$ can be numerically difficult as the constraint set can become the union of disjointed sets.
\end{remark}
}

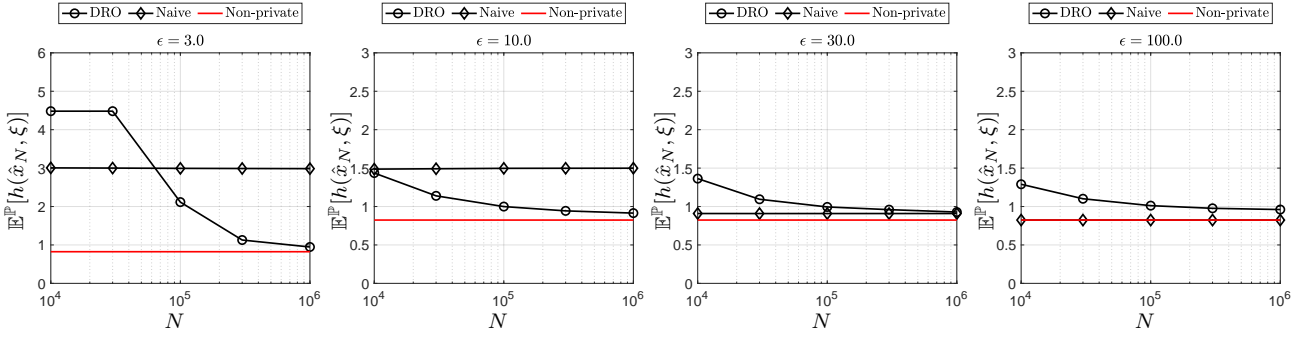
\begin{figure*}
	\centering
	\begin{tikzpicture}
		\node[] at (0,0) {\includegraphics[width=0.25\linewidth]{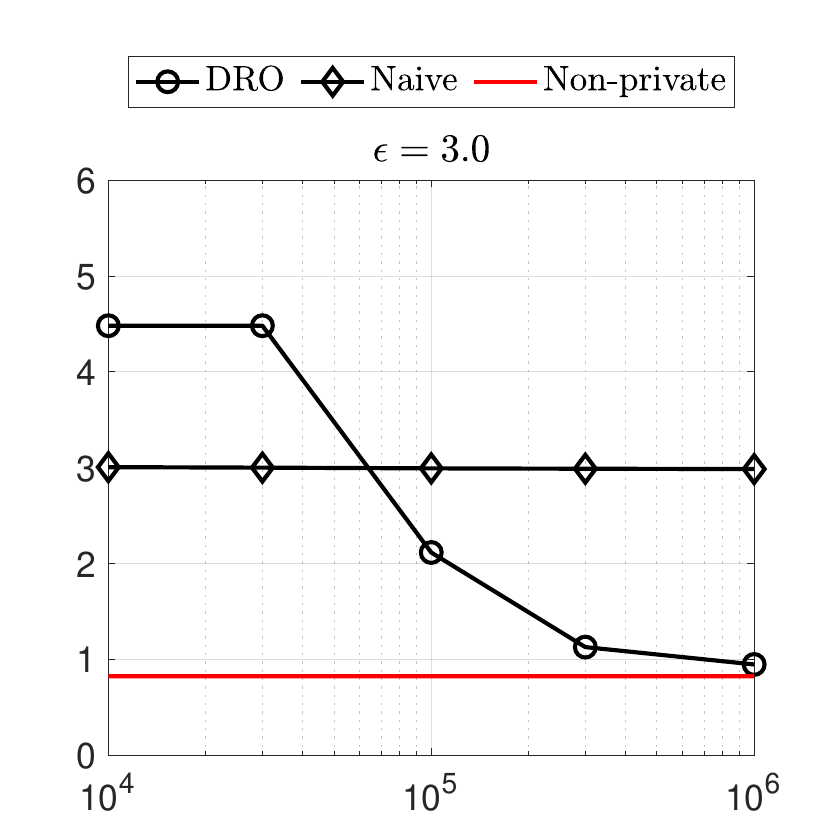}};
		\node[rotate=90] at (-2.1,-0.3) {\footnotesize $\E^\P[h(\hat{x}_N,\xi)]$};
		\node[] at (0,-2.3) {\footnotesize $N$};
		\node[] at (4.3,0) {\includegraphics[width=0.25\linewidth]{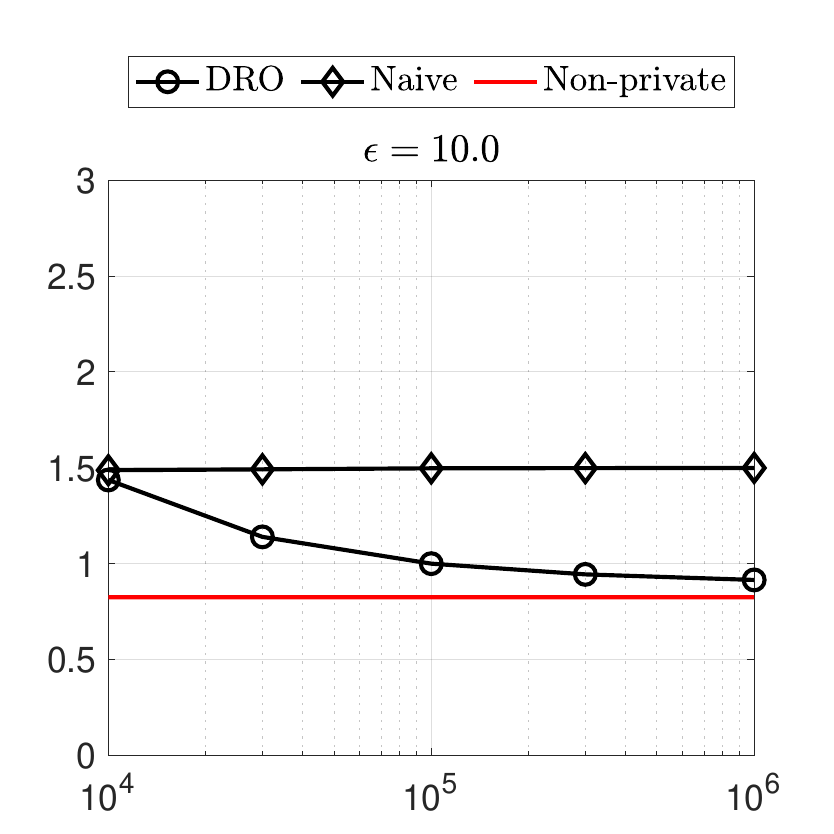}};
		\node[rotate=90] at (2.2,-0.3) {\footnotesize $\E^\P[h(\hat{x}_N,\xi)]$};
		\node[] at (4.3,-2.3) {\footnotesize $N$};
		\node[] at (8.6,0) {\includegraphics[width=0.25\linewidth]{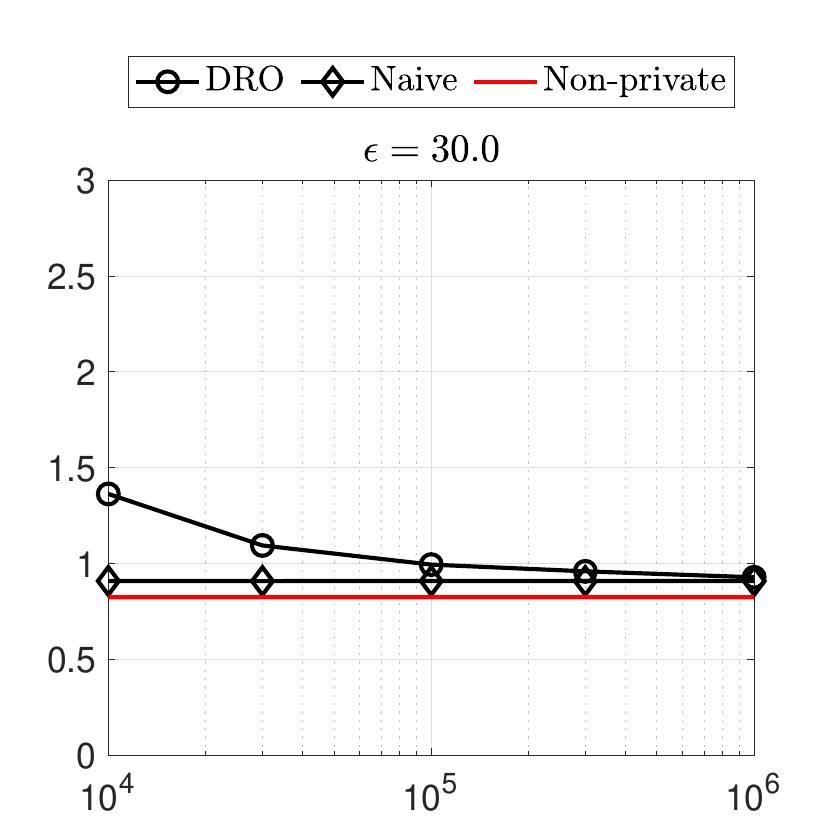}};
		\node[rotate=90] at (6.5,-0.3) {\footnotesize $\E^\P[h(\hat{x}_N,\xi)]$};
		\node[] at (8.6,-2.3) {\footnotesize $N$};
		\node[] at (12.9,0) {\includegraphics[width=0.25\linewidth]{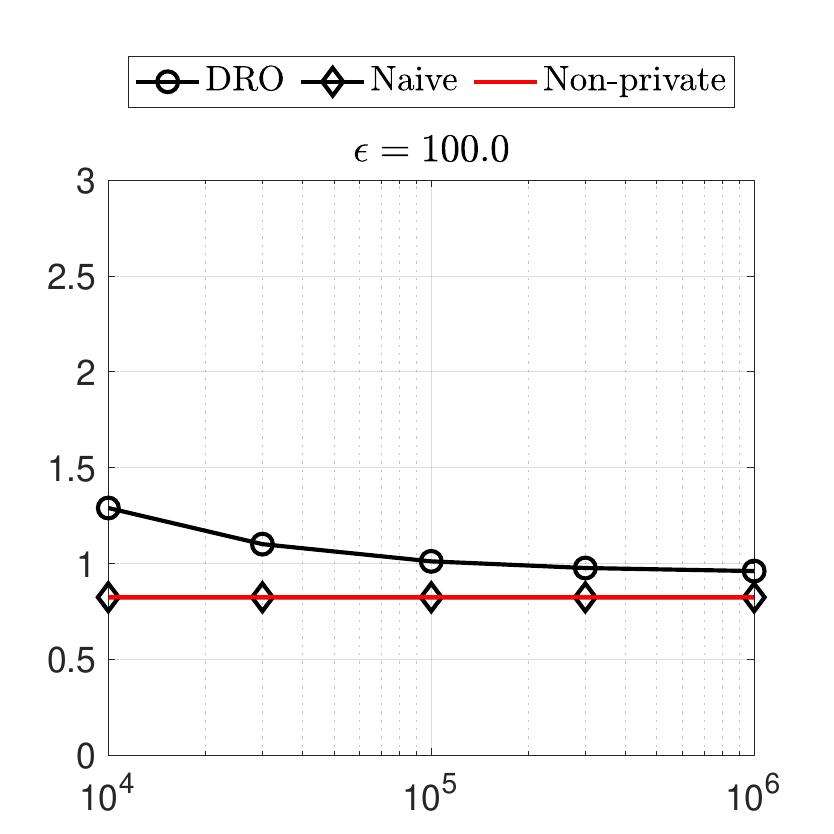}};
		\node[rotate=90] at (10.8,-0.3) {\footnotesize $\E^\P[h(\hat{x}_N,\xi)]$};
		\node[] at (12.9,-2.3) {\footnotesize $N$};
	\end{tikzpicture}
	\vspace{-3mm}
	\caption{Out-of-sample performance $\E^\P[h(\hat{x}_N,\xi)]$ for linear regression with noise-less non-private data (\redline{}), na\"{i}ve linear regression with noisy data (\blackd{}), and distributionally-robust linear regression (\blackcirc{}) versus number of data points $N$.  }
	\vspace{-5mm}
	\label{fig:1}
\end{figure*}

\section{Numerical Example} \label{sec:numerical}
In this section, we demonstrate the capabilities of our results on distributionally-robust optimization with noisy data in the context of privacy-preserving linear regression. We use a dataset containing information regarding nearly 2,260,000 loans made on a peer-to-peer lending platform, called the Lending Club, which is available on Kaggle~\cite{kaggle1}. The dataset contains loan attributes, such as total loan size and interest rates of the loans per annum, and borrower information, such as number of credit lines, state of residence, and age. In our numerical example, we aim to learn a linear regression model for estimating interest rates of the loans based on features of loan size and credit rating. Since our results are for discrete random variables, we discretize the credit rating (mapping scores of 650 to 850 with brackets of 50 to $\{1,\dots,5\}$), the loan amount (mapping \$0 to \$40,000 with brackets of \$10,000 to $\{1,\dots,5\}$), and the interest rate (mapping 5\% to 35\% with increments of 5\% to $\{1,\dots,7\}$). Let $\xi$ be a vector whose entries are, respectively, the discretized credit rating, the discretized loan amount, and the discretized interest rates. To train the linear regression model, we aim to solve~\eqref{eqn:stochastic_program} with $h(x,\xi)=(\xi_3-[\xi_1\;\xi_2\;1]x)^2$. In what follows, however, we assume that we do not have access to the exact measurements of $\xi$. We are in fact supplied with differentially-private perturbed measurements $\xi'$. The following definition and the subsequent theorem make this more clear. Before stating these results, we would like to define the notation $\mathrm{diam}(\Xi)=\max_{\xi,\bar{\xi}} \|\xi-\bar{\xi}\|$.

\begin{definition}[Local Differential Privacy] Conditional probability $\O(\xi'|\xi)$ is $\epsilon$-differentially-private if
$
\revise{\O}[\xi'\in\mathcal{A}|\xi]\leq \exp(\epsilon)\revise{\O}[\xi'\in\mathcal{A}|\bar{\xi}],
$
for all $\mathcal{A}\subseteq \Xi'$ and all $\xi,\bar{\xi}\in\Xi$.
\end{definition}

\begin{proposition}
\label{prop:DP}
Assume $\Xi'=\Xi$. The following conditional probability guarantees $\epsilon$-local differential privacy:
\begin{align*}
\O(\xi'|\xi)={\displaystyle\exp\left(\frac{-\epsilon\|\xi-\xi'\|}{2\mathrm{diam}(\Xi)} \right)}/{\displaystyle\sum_{\xi''\in\Xi}\exp\left(\frac{-\epsilon\|\xi-\xi''\|}{2\mathrm{diam}(\Xi)} \right)}.
\end{align*}
\end{proposition}

\ifdefined\SHORTVERSION
\begin{proof}
	The proof is similar to that of exponential mechanisms for differential privacy in~\cite[\S3.4]{1015610400000042}. Detailed derivations are removed due to space constraints and can be found in an online report~\cite{Farokhi_report_2023}.
\end{proof}
\fi

\ifdefined\LONGVERSION
 \begin{proof} 
 	The proof immediately follows from that
 	 \begin{align*}
 		 \frac{\O(\xi'|\xi)}{\O(\xi'|\bar{\xi})}
 		 =&\frac{\displaystyle\exp\left(\frac{-\epsilon\|\xi-\xi'\|}{2\mathrm{diam}(\Xi)} \right)}{\displaystyle\sum_{\xi''\in\Xi}\exp\left(\frac{-\epsilon\|\xi-\xi''\|}{2\mathrm{diam}(\Xi)} \right)}\\
 		 &\times 
 		 \frac{\displaystyle\sum_{\xi''\in\Xi}\exp\left(\frac{-\epsilon\|\bar{\xi}-\xi''\|}{2\mathrm{diam}(\Xi)} \right)}{\displaystyle\exp\left(\frac{-\epsilon\|\bar{\xi}-\xi'\|}{2\mathrm{diam}(\Xi)} \right)}\\
 		 =&\exp\left(\frac{\epsilon(-\|\xi-\xi'\|+\|\bar{\xi}-\xi'\|)}{2\mathrm{diam}(\Xi)} \right)\\
 		 &\times 
 		 \frac{\displaystyle\sum_{\xi''\in\Xi}\exp\left(\frac{-\epsilon\|\bar{\xi}-\xi''\|}{2\mathrm{diam}(\Xi)} \right)}{\displaystyle\sum_{\xi''\in\Xi}\exp\left(\frac{-\epsilon\|\xi-\xi''\|}{2\mathrm{diam}(\Xi)} \right)}
 		 \\
 		 \leq &\exp\left(\frac{\epsilon\|\bar{\xi}-\xi\|}{2\mathrm{diam}(\Xi)} \right)
 		 \frac{\displaystyle\sum_{\xi''\in\Xi}\exp\left(\frac{-\epsilon\|\bar{\xi}-\xi''\|}{2\mathrm{diam}(\Xi)} \right)}{\displaystyle\sum_{\xi''\in\Xi}\exp\left(\frac{-\epsilon\|\xi-\xi''\|}{2\mathrm{diam}(\Xi)} \right)}\\
 		 =&\exp(\epsilon/2)\exp(\epsilon/2),
 		 \end{align*}
 	 where inequality follows from that $-\|\xi-\xi'\|+\|\bar{\xi}-\xi'\|\leq \|\bar{\xi}-\xi\|$ and that 
 	 \begin{align*}
 		     \frac{\displaystyle\exp\left(\frac{-\epsilon\|\bar{\xi}-\xi''\|}{2\mathrm{diam}(\Xi)} \right)}{\displaystyle\exp\left(\frac{-\epsilon\|\xi-\xi''\|}{2\mathrm{diam}(\Xi)} \right)}
 		     &=
 		     \exp\left(\frac{-\epsilon\|\bar{\xi}-\xi''\|+\epsilon\|\xi-\xi''\|}{2\mathrm{diam}(\Xi)} \right)\\
 		     &\leq \exp\left(\frac{\epsilon\|\bar{\xi}-\xi\|}{2\mathrm{diam}(\Xi)} \right)\\
 		     &\leq \exp(\epsilon/2).
 		 \end{align*}
 	 This concludes the proof.
 \end{proof}
\fi 

For the sake of comparison, we consider three linear regression models. The baseline for the best achievable performance is given by the optimal linear regression model using noise-less data. Note that, by construction, no linear regression model can beat the baseline. However, according to Theorem~\ref{tho:consistency}, the performance of the distributionally-robust linear regression converges to the baseline, i.e., the proposed distributionally-robust regression model is asymptotically consistent and optimal, if $\mathbb{O}$ in Proposition~\ref{prop:DP} is uniformly diagonally dominant. \revise{Due to the special form of  the conditional probability $\O$ in Proposition~\ref{prop:DP}, 
\begin{align*}
    \max_{\xi\neq\xi'}\O(\xi'|\xi)\!\propto\!   
    \displaystyle\exp\!\!\left(\!\frac{-\epsilon \displaystyle\min_{\xi\neq \xi'}\|\xi\!-\!\xi'\|}{2\mathrm{diam}(\Xi)}\! \right)
    \!\!=\!\displaystyle\exp\!\left(\frac{-\epsilon }{2\mathrm{diam}(\Xi)} \right)\!,
\end{align*}
where $\propto$ denotes equality up to re-scaling by a constant or proportionality (c.f., Proposition~\ref{prop:DP}) and $\min_{\xi\neq \xi'}\|\xi-\xi'\|=1$ in this example. Therefore, Assumption~\ref{assum:diagonally_dominant} is satisfied if
\begin{align*}
    \frac{\O(\xi|\xi)}{\max_{\xi\neq \xi'}\O(\xi'|\xi)}
    =\exp\!\left(\frac{\epsilon }{2\mathrm{diam}(\Xi)} \right)>|\Xi|,
\end{align*}
or equivalently if $\epsilon>2\mathrm{diam}(\Xi)\log(|\Xi|)\approx 64.17$.} Note that, although we consider privacy budgets $\epsilon$ that are below this bound and thus the conditional density of the privacy-preserving noise is not uniformly diagonally dominant, we can still observe asymptotic consistency numerically. We compare the baseline performance with the performance of two linear regression models in the noisy regime. One of the regression models is \revise{na\"{i}vely} constructed from the noisy data without any processing (i.e., as if the data was noiseless). The other model is the distributionally-robust regression model that can be extracted from Corollary~\ref{cor:worst_case}. \revise{This optimization problem is modelled using CVX~\cite{cvx} and solved using SeDuMi~\cite{sturm1999using}. The codes for conducting the experiments in this section can be downloaded from GitHub\footnote{https://github.com/farhadfarokhi/NoisyDRO}.
}

Figure~\ref{fig:1} illustrates the out-of-sample performance $\E^\P[h(\hat{x}_N,\xi)]$ for linear regression with noise-less non-private data, na\"{i}ve linear regression with noisy data, and distributionally-robust linear regression with noisy data versus the number of data points $N$. For $\epsilon=3.0$, \revise{the out-of-sample performance of the distributionally-robust regression model improves rapidly and surpass  the na\"{i}ve regression model}. For $\epsilon=10.0$, the out-of-sample performance of the distributionally-robust regression model is superior to the na\"{i}ve regression model for the entire range. \revise{Finally, for $\epsilon=30.0$ $\epsilon=100.0$, due to the  small magnitude of the noise, the out-of-sample performance of the na\"{i}ve regression model and linear regression with noise-less non-private data are almost identical, and the out-of-sample performance of the distributionally-robust regression model approaches that of the noise-less regression model rapidly.  In all cases, as the number of data points $N$ grows, the out-of-sample performance of the distributionally-robust regression model approaches that of the noise-less regression model.}

\section{Conclusions and Future Work}
\label{sec:conclusions}
We considered stochastic programs where the uncertainty distribution must be inferred from noisy data samples. We showed that the stochastic programs can be approximated with distributionally\-/robust optimizations that minimize the worst-case expected cost over an ambiguity set of distributions that are sufficiently compatible with the observed data. Future work can focus on continuous random variables.

\ifdefined\LONGVERSION
\appendix
\section{Useful Lemma}

\begin{lemma} \label{lemma:useful}
	Under Assumption~\ref{assum:diagonally_dominant}, there exists $c_0(\O)\in(0,\infty)$ such that $\d_{\rm TV}(\Q,\P) \leq c_0(\O)\d_{\rm TV}(\O\star\Q,\O\star\P).$
\end{lemma}

\begin{proof}
 Note that
 \begin{align*}
	 	|(\O\star\Q)(\xi')&-(\O\star\P)(\xi')|\\
	 	=&\Bigg|\sum_{\xi\in\Xi} \O(\xi'|\xi) \Q(\xi)-\sum_{\xi\in\Xi} \O(\xi'|\xi) \P(\xi)\Bigg|\\
	 	=&\Bigg|\sum_{\xi\in\Xi} \O(\xi'|\xi) (\Q(\xi)-\P(\xi))\Bigg|\\
	 	=&\Bigg|\O(\xi'|\xi') (\Q(\xi')-\P(\xi')) \\
	            &\quad\quad+ \sum_{\xi\in\Xi,\xi\neq \xi'} \O(\xi'|\xi) (\Q(\xi)-\P(\xi))\Bigg|\\
	 	\geq & \O(\xi'|\xi') |\Q(\xi')-\P(\xi')|\\
	            &-\Bigg| \sum_{\xi\in\Xi,\xi\neq \xi'} \O(\xi'|\xi) (\Q(\xi)-\P(\xi))\Bigg|\\
	 	\geq & \min_{x\in\Xi}\O(x|x) |\Q(\xi')-\P(\xi')|\\
	            &-\Bigg| \sum_{\xi\in\Xi,\xi\neq \xi'} \O(\xi'|\xi) (\Q(\xi)-\P(\xi))\Bigg|,
	 \end{align*}
 and, as a result,
 \begin{align*}
	 	\min_{x\in\Xi}\O(x&|x) \sum_{\xi'\in\Xi}|\Q(\xi')-\P(\xi')|\\
	 	\leq &\sum_{\xi'\in\Xi} |(\O\star\Q)(\xi')-(\O\star\P)(\xi')|\\
	 	&+\sum_{\xi'\in\Xi}\Bigg| \sum_{\xi\in\Xi,\xi\neq \xi'} \O(\xi'|\xi) (\Q(\xi)-\P(\xi))\Bigg|\\
	 	\leq &\sum_{\xi'\in\Xi} |(\O\star\Q)(\xi')-(\O\star\P)(\xi')|\\
	 	&+\sum_{\xi'\in\Xi}\sum_{\xi\in\Xi,\xi\neq \xi'} \O(\xi'|\xi) |\Q(\xi)-\P(\xi)|\\
	 	\leq &\sum_{\xi'\in\Xi} |(\O\star\Q)(\xi')-(\O\star\P)(\xi')|\\
	 	&+\max_{\xi\in\Xi,\xi\neq \xi'} \O(\xi'|\xi)\sum_{\xi'\in\Xi}\sum_{\xi\in\Xi,\xi\neq \xi'} |\Q(\xi)-\P(\xi)|\\
	 	\leq &\sum_{\xi'\in\Xi} |(\O\star\Q)(\xi')-(\O\star\P)(\xi')|\\
	 	&+\max_{\xi,\xi'\in\Xi,\xi\neq \xi'} \O(\xi'|\xi)\sum_{\xi'\in\Xi}\sum_{\xi\in\Xi} |\Q(\xi)-\P(\xi)|\\
	 	\leq &\sum_{\xi'\in\Xi} |(\O\star\Q)(\xi')-(\O\star\P)(\xi')|\\
	 	&+|\Xi|\max_{\xi,\xi'\in\Xi,\xi\neq \xi'} \O(\xi'|\xi)\sum_{\xi\in\Xi} |\Q(\xi)-\P(\xi)|.
	 \end{align*}
 Rearranging the terms gives
 \begin{align*}
	 	c'_0(\O)\d_{\rm TV}(\Q,\P) \leq \d_{\rm TV}(\O\star\Q,\O\star\P),
	 \end{align*}
 where
 \begin{align*}
	 	c'_0(\O):=\displaystyle \min_{x\in\Xi}\O(x|x)-|\Xi|\max_{\xi,\xi'\in\Xi,\xi\neq \xi'} \O(\xi'|\xi).
	 \end{align*}
 Note that $c'_0(\O)>0$ so long as Assumption~\ref{assum:diagonally_dominant} holds. The rest follows by setting $c_0(\O)=1/c'_0(\O)$. 
\end{proof}

\fi

\bibliography{ref}
\bibliographystyle{ieeetr}

\end{document}